\documentclass[11pt]{article}
\usepackage{amsmath,amsfonts,amsthm}
\usepackage{amssymb}

\newtheorem{theorem}{Theorem}
\newtheorem{lemma}[theorem]{Lemma}
\newtheorem{corollary}[theorem]{Corollary}
\newtheorem{proposition}[theorem]{Proposition}

\newtheorem{conjecture}[theorem]{Conjecture}

\begin{document}

\title{Notes on Cops and Robber game on graphs}
\author{Bojan Mohar\\[2mm]
     Simon Fraser University \& IMFM\\
     {\tt mohar@sfu.ca}}

\date{}

\maketitle

\begin{abstract}
These are some personal notes about the pursuit game of Cops and Robbers that I made starting in 2007. More old and new problems (and some solutions) will be added in future versions of these notes.
\end{abstract}



I learned about Cops and Robbers from Brian Alspach in 2006. He asked me about the result by Schroeder \cite{Schroder} who proved that the cop number of a graph of genus $g$ is at most $\lfloor \tfrac{3}{2}g\rfloor+3$ and asked whether the factor $\tfrac{3}{2}$ can be improved to $1$.  Further I got interested in this area when G\'abor Kun was visiting SFU as a postdoctoral fellow of G\'abor Tardos. He told me about his results with Bollob\'as and Leader \cite{BollobasKunLeader} on the cop number of random graphs.

\subsection*{Notation}

Throughout this note we use standard terminology and notation. The following in particular:
\begin{itemize}
\item
$c(G)$ is the cop number of $G$,
\item
$\delta(G)$ and $\Delta(G)$ are the minimum and the maximum degree,
\item
$g(G)$ and $\widetilde g(G)$ are the genus and the nonorientable genus (crosscap number) of the graph,
\item
$c(g) = \max \{c(G)\mid g(G)\le g\}$ and $\widetilde c(g) = \max \{c(G)\mid \widetilde g(G)\le g\}$.
\end{itemize}

\section{Schroeder's Conjecture}

One of the most inspiring results about the Cops and Robber game is the fact that the cop number of any planar graph is at most 3. This can be generalized to graphs of any fixed genus $g$. It is relatively easy to see by taking a shortest homologically non-trivial cycle $C$ in a graph $G$ embedded on a surface of genus $g$ (or nonorientable genus $g$) that two cops can guard the cycle and the remaining graph has smaller genus. This implies that
$$c(G)\le 2g+3.$$
Schr\"oder improved this bound to the following:

\begin{theorem}[Schr\"oder \cite{Schroder}]
Let $G$ be a graph of genus $g$. Then $c(G)\le\tfrac{3}{2}g + 3$.
\end{theorem}

He conjectured that the upper bound can be improved:

\begin{conjecture}[Schr\"oder \cite{Schroder}]
Let $G$ be a graph of genus $g$. Then $c(G)\le g + 3$.
\end{conjecture}

The conjecture holds for $g=1$ (the torus case), but as of today no toroidal graph with cop number 4 is known.

Most people I talked to believe that Schr\"oder's conjectured bound is tight up to a constant. However, there is no evidence why this would be true. One purpose of this note is to discuss lower bounds and provoke a stronger conjecture.

\section{Lower bounds}

A simple lower bound on the cop number can be derived if the graph has girth at least 5.

\begin{lemma}
If $girth(G)\ge5$, then $c(G)\ge \delta(G)$.
\end{lemma}

\begin{proof}
We say that a cop \emph{guards} a vertex $u$ if he is positioned either at $u$ or at a neighbor of $u$.
The robber has a simple escape strategy if there are fewer than $d=\delta(G)$ cops. Let $v\in V(G)$ be the current position of the robber. Since $G$ has girth at least 5, each cop can guard at most one neighbor of $v$, so there is an unguarded neighbor of $v$, where the robber can go without being caught.
\end{proof}

The natural question arises what is the largest $d=d(g)$ for which there is a graph of genus at most $g$ and girth at least 5. To answer this question, we first show that the genus of dense graphs is proportional to the number of edges.

\begin{proposition}\label{prop:genus proportional to e}
Let $G$ be a connected graph of order $n$ with $\alpha n$ edges, where $\alpha > 3$. Then
$$
  \tfrac{1}{6}(\alpha-3)n < g(G) \le (\alpha - 1)n
$$
and
$$
  \tfrac{1}{3}(\alpha-3)n < \widetilde g(G) \le (\alpha - 1)n.
$$
\end{proposition}

\begin{proof}
The upper bound is easy. First we embed a spanning tree plus one edge in the plane and then add edge by edge, each time adding a new handle if necessary to extend the embedding. The proof in the nonorientable case is similar. We keep the property that at any time there are at most two faces. If the added edge splits a single face into two, then we add a crosscap on it, and if the edge to be added needs to start in one face and cross to the second one, we can add a crosscap at the crossing, thus making a single new face, and then the edge may split this face into two or keep just one face.

For the lower bound we just use Euler's formula (see \cite{MT01}) which says that $2g \ge 2 - n + e - f$, where $e$ is the number of edges and $f$ is the number of faces. Having a simple graph, each face boundary has at least 3 edges and hence $f\le \tfrac{2}{3}e$, giving
$$
   2g \ge 2 - n + \tfrac{1}{3}e = 2-n+\tfrac{1}{3}\alpha n > \tfrac{1}{3}(\alpha-3)n.
$$
The same proof works for the nonorientable genus except that we have to replace $2g$ with $g$ in Euler's formula.
\end{proof}

\begin{lemma}
If $girth(G)\ge5$, then $g(G)\ge \tfrac{1}{10}(3\delta(G)^3-10\delta^2) = \Omega(\delta^3)$.
\end{lemma}

\begin{proof}
Without repeating the proof of Proposition \ref{prop:genus proportional to e}, we note that the lower bound in case of graphs of girth 5 can be improved. In this case, $f\le \tfrac{2}{5}e$, and we obtain:
$$
    g(G) \ge (\tfrac{3\delta}{10}-1)n.
$$
Of course, $n\ge 1 + \delta^2$, which implies the bound of the lemma.
\end{proof}

It is possible to construct graphs of genus $g$ and girth 5 whose minimum degree is proportional to $g^{1/3}$ and whose number of vertices is proportional to $g^{2/3}$. (Simple examples of this kind are incidence graphs of finite projective planes.) This shows that
$$
   c(g) \ge \Omega(g^{1/3}).
$$
For a while this was best we could get, but there is a better lower bound. For this bound we shall use random graphs (see \cite{BollobasKunLeader} or \cite{LuczakPralat}).

\begin{theorem}[Bollob\'as, Kun, and Leader \cite{BollobasKunLeader}]
\label{thm:random}
Suppose that $p=p(n)\ge 2.1\log(n)/n$. Then a.a.s.
$$
   (np)^{-2} n^{1/2-o(1)} \le c(G_{n,p}) \le 160000 \sqrt{n} \log(n).
$$
\end{theorem}

\begin{corollary}
$c(g)\geq g^{\frac{1}{2}(1-o(1))}$. More precisely, for every $\varepsilon>0$ and every sufficiently large $g$, there is a graph $G$ of genus $g$ with cop number bounded:
$$g^{\frac{1}{2}-\varepsilon} \le c(G) \le g^{\frac{1}{2}+\varepsilon}.$$
\end{corollary}

\begin{proof}
Let $p=p(n)=\tfrac{5}{2}\log(n)/n$. The random graph $G=G(n,p)$ has less than $2n\log n$ edges with high probability and thus its genus $g$ is smaller than $2n\log n$. By Theorem \ref{thm:random}, its cop number is at least
$$c(G) \ge (np)^{-2} n^{1/2-o(1)} \ge g^{\frac{1}{2}-\varepsilon}$$
w.h.p. if $n$ is large enough.

The upper bound holds for the same graph. The proof uses the same two results (just the opposite bounds), where the bigger constants and the $\log(n)$ factors can be hidden in $g^{\varepsilon}$ when $g$ is large enough.
\end{proof}

\subsection{A tighter conjecture}

The discussed lower bounds led me to the following conjecture that has been unchallenged as of today.

\begin{conjecture}[Mohar, 2009]
$c(g)= g^{\frac{1}{2}+o(1)}$\ and \ $\widetilde c(g)= g^{\frac{1}{2}+o(1)}$.
In other words, for every $\varepsilon >0$ there exists $g_0$ such that for every $g\ge g_0$,
$$g^{\frac{1}{2}-\varepsilon} < c(g) < g^{\frac{1}{2}+\varepsilon}\quad \textrm{and} \quad
g^{\frac{1}{2}-\varepsilon} < \widetilde c(g) < g^{\frac{1}{2}+\varepsilon}.$$
\end{conjecture}

\section{Cops and Robbers on Riemannian surfaces}

Let $S$ be a Riemannian surface. Again, we can place $k$ cops and a robber on $S$ and ask whether the cops can capture the robber if they all move with velocity at most $1$ (during time $t$, any cop and the robber can move along a path of length at most $t$) and the goal is to get a cop at distance less than $\tfrac{1}{2}$ from the robber (when we declare that the robber is caught).
The cops and the robber know positions of each other at any given time.

There are two versions of the game:

\begin{itemize}
\item {\bf Continuous strategy}: All players move continuously.
\item {\bf Discrete moves}: Cops move during the day (one time unit), when the robber sleeps, and stay alert during the night, guarding their neighborhood up to distance $\tfrac{1}{2}$ from their position. The robber moves during the night only.
\end{itemize}

The continuous version needs some justification what a strategy means. Is it just the limit when the time units of discrete game tend to 0?

Here are some very basic questions:

\begin{itemize}
\item
Is the cop number of surfaces bounded in terms of the genus?
\item
If the cops have a strategy to catch the robber, could it in fact be achieved that one of the cops comes arbitrarily close to the robber (or even onto the same point) in the continuous version of the game?
\item
Which surfaces of a fixed genus are the worst?
\item
What is the supremum of the cop numbers taken over all surfaces whose cut locus is at least 1 and whose area is at most $\alpha$? Is it bounded?
\item
Are the surfaces of constant curvature any simpler than general Riemannian surfaces with respect to the Cops and Robber game?
\end{itemize}

There are obvious higher-dimensional analogues.

\medskip

The basic strategy of cops (the geodesic path lemma) seems to work in this setting as well.

\begin{lemma}
Let $I$ be a geodesic path in $S$. Then one cop can guard $I$ (after the cop reaches $I$ and spends time equal to the length of $I$ on the path to adjust himself, whenever the robber steps on $I$ or crosses it, he will be caught by the cop).
\end{lemma}

\begin{proof}
The proof is the same as in the case of geodesic paths in graphs. Let $L$ be the length of $I$. Then we define, for each point $s\in S$, its \emph{shadow} $p_I(s)\in I$ as follows. Let $a,b$ be the ends of $I$. If $dist(s,a) \ge L$, then we set $p_I(s)=b$. Otherwise, we let $p_I(s)$ be the point on $I$ whose distance from $a$ is equal to $dist(s,a)$. Initially the cop moves to $a$ and then progresses towards $b$ until he reaches the shadow of the robber. From that point on, he stays at the shadow all the time. This strategy works well in the continuous and in the discrete version.
\end{proof}

\end{document}